\documentclass[12pt]{article}
\usepackage{amsmath, amsthm, amsfonts}
\usepackage{amssymb}
\usepackage{amscd}
\usepackage{verbatim}

\begin{document}
\newcommand{\TM}{\stackrel{\cdot}{\to}}
\newcommand{\half}{\frac12}
\newcommand{\eqdef}{\stackrel{{\mathrm def}}{=}}
\newcommand{\M}{{\mathcal M}}
\newcommand{\loc}{{\mathrm{loc}}}
\newcommand{\dx}{\,\mathrm{d}x}
\newcommand{\core}{C_0^{\infty}(\Omega)}
\newcommand{\sob}{W^{1,p}(\Omega)}
\newcommand{\sobloc}{W^{1,p}_{\mathrm{loc}}(\Omega)}
\newcommand{\merhav}{{\mathcal D}^{1,p}}
\newcommand{\be}{\begin{equation}}
\newcommand{\ee}{\end{equation}}
\newcommand{\mysection}[1]{\section{#1}\setcounter{equation}{0}}
\newcommand{\bea}{\begin{eqnarray}}
\newcommand{\eea}{\end{eqnarray}}
\newcommand{\bean}{\begin{eqnarray*}}
\newcommand{\eean}{\end{eqnarray*}}
\newcommand{\thkl}{\rule[-.5mm]{.3mm}{3mm}}
\newcommand{\cw}{\stackrel{D}{\rightharpoonup}}
\newcommand{\id}{\operatorname{id}}
\newcommand{\supp}{\operatorname{supp}}
\newcommand{\wlim}{\mbox{ weak-lim }}
\newcommand{\mymu}{{x_N^{-p_*}}}
\newcommand{\R}{{\mathbb R}}
\newcommand{\N}{{\mathbb N}}
\newcommand{\Z}{{\mathbb Z}}
\newcommand{\Q}{{\mathbb Q}}
\newcommand{\abs}[1]{\lvert#1\rvert}
\newtheorem{theorem}{Theorem}[section]
\newtheorem{corollary}[theorem]{Corollary}
\newtheorem{lemma}[theorem]{Lemma}
\newtheorem{definition}[theorem]{Definition}
\newtheorem{remark}[theorem]{Remark}
\newtheorem{proposition}[theorem]{Proposition}
\newtheorem{assertion}[theorem]{Assertion}
\newtheorem{problem}[theorem]{Problem}
\newtheorem{conjecture}[theorem]{Conjecture}
\newtheorem{question}[theorem]{Question}
\newtheorem{example}[theorem]{Example}
\newtheorem{Thm}[theorem]{Theorem}
\newtheorem{Lem}[theorem]{Lemma}
\newtheorem{Pro}[theorem]{Proposition}
\newtheorem{Def}[theorem]{Definition}
\newtheorem{Exa}[theorem]{Example}
\newtheorem{Exs}[theorem]{Examples}
\newtheorem{Rems}[theorem]{Remarks}
\newtheorem{Rem}[theorem]{Remark}

\newtheorem{Cor}[theorem]{Corollary}
\newtheorem{Conj}[theorem]{Conjecture}
\newtheorem{Prob}[theorem]{Problem}
\newtheorem{Ques}[theorem]{Question}
\newcommand{\pf}{\noindent \mbox{{\bf Proof}: }}
\newcommand{\Lag}{{\mathcal L}}

\renewcommand{\theequation}{\thesection.\arabic{equation}}
\catcode`@=11 \@addtoreset{equation}{section} \catcode`@=12


\title{On a version of Trudinger-Moser inequality with M\"obius shift invariance}
\author{Adimurthi\\
 {\small Centre of Applicable Mathematics}\\ {\small
 Tata Institute of Fundamental Research}\\
 {\small P.O.Box No. 1234}\\
{\small Bangalore - 560 012, India}\\
{\small aditi@aditi@math.tifrbng.res.in}\\\and Kyril Tintarev
\\{\small Department of Mathematics}\\{\small Uppsala University}\\
{\small P.O.Box 480}\\
{\small SE-751 06 Uppsala, Sweden}\\{\small
kyril.tintarev@math.uu.se}}
\maketitle
\begin{abstract}
The paper raises a question about the optimal critical nonlinearity for the Sobolev space in two dimensions, connected to loss of compactness, and discusses the pertinent concentration compactness framework.
We study properties of the improved version of the Trudinger-Moser inequality on the open unit disk $B\subset\R^2$, recently proved by G. Mancini and K. Sandeep \cite{Sandeep}. Unlike the original Trudinger-Moser inequality, this inequality is invariant with respect to M\"obius automorphisms of the unit disk, and as such is a closer analogy of the critical nonlinearity $\int |u|^{2^*}$ in the higher dimension than the original Trudinger-Moser nonlinearity.
\\[2mm]
\noindent  2000  \! {\em Mathematics  Subject  Classification.}
Primary  \! 35J20, 35J60; Secondary \! 46E35, 47J30, 58J70.\\[1mm]
\end{abstract}


\mysection{Introduction}
In this paper we study an inequality that improves the classical (Pohozhaev)-Trudinger-Moser inequality (\cite{Pohozhaev}, \cite{Trudinger}, \cite{Moser}) on a unit disk $B$  in $\R^2$:
\be
\label{TM}
\sup_{u\in H_0^1(B),\|\nabla u\|_2\le 1}\int_{B} e^{4\pi u^2}dx<\infty.
\ee
The result below has been recently proved by Mancini and Sandeep \cite{Sandeep} (for the invariant formulation in terms of hyperbolic space $\mathbb H^2$ see Theorem~\ref{thm:HTM+} below.)
\begin{theorem}
 \label{thm:WTM} Let $B$ be an open unit disk in $\R^2$. The following relation holds true:
\be
\label{WTM}
\sup_{u\in H_0^1(B),\|\nabla u\|_2\le 1}\int_{B} \frac{e^{4\pi u^2}-1}{(1-|x|^2)^2}dx<\infty.
\ee
\end{theorem}
We give a different proof to this inequality, based on coverings defined by M\"obius transformations, rather than on rearrangements on the hyperbolic space like in \cite{Sandeep}, which provides insights for further results.
An elementary corollary of this inequality is the Trudinger-Moser inequality in the exterior of the unit disk which immediately follows from \eqref{WTM} by the change of variable $x\mapsto x/|x|^2$.
\begin{corollary}
  Let $B$ be an open unit disk in $\R^2$. The following relation holds true:
\be
\label{WTMe}
\sup_{u\in \mathcal D^{1,2}_0(\R^2\setminus B),\|\nabla u\|_2\le 1}\int_{\R^2\setminus B} \frac{e^{4\pi u^2}-1}{(|x|^2-1)^2}dx<\infty.
\ee
\end{corollary}

For the Sobolev space $H^1$ in two dimensions, the (Pohozhaev)-Trudinger-Moser functional $\int e^{4\pi u^2}dx$ is widely accepted as a standard nonlinearity of critical growth, that is, as a counterpart of $\int_{\R^N}|u|^\frac{2N}{N-2}dx$ in the case $N>2$. This view is justified by the following analogy. When $N>2$, the functional $\int_{\R^N}|u|^pdx$ is continuous in $H^1(\R^N)$ when $p\in(2,2^*]$, $2^*\eqdef \frac{2N}{N-2}$, and it is unbounded on any bounded subset of $H^1(\R^N)$ when $p>2^*$. In the case $N=2$, the functional  $\int_{\R^2} e^{p u^2}dx$ on the set $\{u\in H^1(\R^2),\| u\|_{H^1}\le 1\}$ is bounded if  and only if $p\le 4\pi$ (see \cite{Ruf}). The analogy extends also to weak continuity properties. For obvious reason of translation invariance, there is no weak continuity if the domain of integration is the whole $\R^N$. If, however, $\Omega\subset\R^N$, $N>2$, is a bounded domain, then the functional $\int_{\Omega}|u|^pdx$ is weakly continuous on $H^1_0(\Omega)$ whenever  $p<2^*$, and, similarly, if, $N=2$, the functional 
 $\int_{\Omega} e^{p u^2}dx$ is weakly continuous on $\{u\in H_0^1(\Omega),\| u\|_{H_0^1}\le 1\}$ whenever $p<4\pi$.

This analogy, however, does not extend to the critical nonlinearities, $p=2^*$ resp. $p=4\pi$. When  $N>2$,  the functional  $\int_{\Omega}|u|^{2^*}dx$ is not weakly continuous at any point, but if $N=2$, the functional $\int_{\Omega} e^{4\pi u^2}dx$ is sequentially weakly continuous at every point of $\{u\in H_0^1(\Omega),\|\nabla u\|_2\le 1\}\setminus\{0\}$ (see \cite{PLL2b}).

Lack of compactness for the critical nonlinearity in the case $N>2$ can be traced to the following symmetries of the space
$\mathcal D^{1,2}(\R^N)$  (defined as the completion of $C_0^\infty$ with respect to the gradient norm $\|\nabla \cdot\|_2$):
$$
D_N=\{g_{s,y}u(x)=2^{\frac{N-2}{2}s}u(2^s(x-y)), s\in\R, y\in\R^N\},
$$
that is, to actions of translations and dilations.
These operators are linear isometries on both $\mathcal D^{1,2}(\R^N)$ and $L^{2^*}(\R^N)$, so that for every $u\in \mathcal D^{1,2}(\R^N)$, $u_k\eqdef g_{s_k,y_k}u\rightharpoonup 0$ whenever $|s_k|+|y_k|\to\infty$, while the respective norms of $u_k$ coincide with that of $u$. A similar counterexample suitable for a bounded domain  is given by $y_k=0$, $s_k\to+\infty$ and $u$ supported on a convex compact subset.
Lack of compactness in the imbedding of $\mathcal D^{1,2}(\R^N)\subset L^{2^*}(\R^N)$ can be, in some sense attributed entirely to the group $D_N$, namely, the compactness is restored if one ``factors out" the action of the
group (see e.g. Lemma~5.3, \cite{ccbook}):
$$
\forall s_k\in\R, y_k\in\R^N, g_{s_k,y_k}u_k\rightharpoonup 0 \text{ in }  \mathcal D^{1,2}(\R^N) \Rightarrow u_k\to 0 \text{ in } L^{2^*}(\R^N).
$$

Weak continuity properties of the critical nonlinearity in the case $N=2$ indicate that there is no known non-compact group, other than Euclidean shifts, that preserves both the Sobolev norm and the Trudinger-Moser nonlinearity $\int_{\R^2} e^{4\pi u^2}dx$. The matter is further complicated by the fact that in this case there is no dilation-invariant functional space $\mathcal D^{1,2}$: the completion of $C_0^\infty(\R^2)$ with respect to the gradient norm lacks continuous imbedding even into the space of distributions. On the other hand, the problem in the space $H_0^1(B)$ (which we in what follows consider equipped with the equivalent Sobolev norm $\|\nabla u\|_2$), admits two groups of linear unitary operators, defined below, that play a role similar, respectively, to actions of dilations and of translations.  

The Trudinger-Moser functional $\int_B e^{4\pi u^2}dx$ fails to be invariant with respect to either of these groups.
This, however, happens to testify not for irrelevance of these groups but for an observation that the Trudinger-Moser functional is not the sharp critical nonlinearity and can be replaced by a stronger expression. It remains an open problem, however, to find a sharp critical nonlinearity that is invariant with respect to the product group. The details are as follows.

\subsection{Dilation-invariant nonlinearity}
In this paragraph we summarize results of  \cite{AdiTindoO}.
Let $H_{0,r}^1(B)$ denote the subspace of radial functions of  $H_{0}^1(B)$.
The transformations
\begin{equation}
\label{dilations}
h_su(r)\eqdef s^{-\frac12}u(r^s), \, u\in H_{0,r}^1(B) \,s>0,
\end{equation}
preserve the norm  $\|\nabla u\|_2$ of $H_{0,r}^1(B)$, as well as the 2-dimensional Hardy functional
$\int_B\frac{u^2}{|x|^2(\log 1/|x|)^2}dx$ (for the Hardy inequality in dimension 2 see Adimurthi and Sandeep \cite{AdiSandeep}). Furthermore, these transformation preserve the norms of a family of weighted $L^p$-spaces, $p=[2,\infty]$, analogous to the weighted-$L^p$ scale with $p\in[2,2^*]$ produced by H\"older inequality in the case $N>2$, interpolating between the Hardy term $\int\frac{u^2}{|x|^2}dx$ and the critical nonlinearity $\int |u|^{2^*}dx$.
In the case $N=2$, the critical exponent is formally $2^*=+\infty$ and the dilation-invariant $L^{2^*}$-norm is
\be
\label{infnorm}
\|u\|_\infty=\sup_{r\in(0,1)}\frac{|u(r)|}{(\log\frac{1}{r})^{1/2}}.
\ee

The Trudinger-Moser nonlinearity is not, however, dilation-invariant. On the other hand it is continuous with respect to the norm \eqref{infnorm}, which means that  the $L^\infty$- nonlinearity \eqref{infnorm} gives a sharp, dilation-invariant improvement of the Trudinger-Moser nonlinearity, even if only for the subspace of radial functions.

\subsection{M\"obius transformations}
We refer the reader to the Appendix for definitions and basic properties connected to M\"obius transformations and the Poincar\'e disk.
Adopting, for the sake of convenience, the complex numbers notation $z=x_1+ix_2$ for points $(x_1,x_2)$ on the plane, we consider the following subset of automorphisms of the unit disk, known as M\"obius transformations.
\be
\label{eta0}
\eta_\zeta (z)= \frac{z-\zeta}{1-\bar\zeta z}, \zeta\in B.
\ee
Since the maps \eqref{eta0} are conformal automorphisms of the unit disk, one has $|\nabla u\circ \eta_\zeta|_2=|\nabla u|_2$, which implies that the M\"obius shifts $u\mapsto u\circ\eta_\zeta$, $\zeta\in B$, are unitary operators in $H_0^1(B)$.
The gradient norm on the disk is the coordinate representation of the quadratic form of Laplace-Beltrami operator on $\mathbb H^2$ regarded as the Poincar\'e disk model, which allows to identify $H_0^1(B)$ as representation of the space $\dot H^1(\mathbb H^2)$, defined by completion of $C_0^\infty(\mathbb H^2)$ with respect to the gradient norm. M\"obius shifts give rise therefore to unitary operators on $\dot H^1(\mathbb H^2)$.

Furthermore, under the Poincar\'e disk model, the maps \eqref{eta0} define isometries on  $\mathbb H^2$. Consequently, we have nonlinearities on the unit disk, invariant with respect to M\"obius shifts, of the form (in the manifold notation and in the terms of Poincar\'e disk):
$$\int_{\mathbb H^2}F(u)d\mu=\int_BF(u)\frac{dx}{(1-|x|^2)^2}.$$
In particular, once the inequality \eqref{WTM} is verified, the functional $\int_B\frac{e^{4\pi u^2}}{(1-|x|^2)^2}dx$ possesses both critical growth and invariance with respect to M\"obius shifts.
\subsection{Main results}
In addition to Theorem~\ref{thm:WTM}, which trivially follows from its hyperbolic space counterpart, Theorem~\ref{thm:HTM+}  proved in Section~2, we study weak continuity properties of subcritical (but not weakly continuous) nonlinearities of the form $\int_BF(u)\frac{dx}{(1-|x|^2)^2}$,  and existence of maximizers for a related isoperimetric problem.  We prove
\begin{theorem}
\label{thm:coco}
Let $F\in C(\R)$ satisfy, with some $C>0$, $r>2$ and $p<4\pi$,
\be
\label{cocoF}
|F(s)|\le C|s|^re^{ps^2}.
\ee
If  $u_k\in H_0^1(B)$, $\|\nabla u_k\|_2=1$, satisfies the condition
\be
\label{cw}
\text{ For every sequence }  \zeta_k\in B,u_k\circ\eta_{\zeta_k}\rightharpoonup 0,
\ee
then
\be
\label{cocodmu}
\int_BF(u_k)d\mu\to 0.
\ee
\end{theorem}

This theorem is required for the following existence result.
\begin{theorem}
\label{minimizers}
Let $F\in C^1(\R)$, $\sup F>0$, satisfy  \eqref{cocoF} with some $C>0$, $r>2$ and $p<4\pi$.
If, in addition, for every $t\in(0,1)$ and  $a,b\in\R$,
\be
\label{convexity}
F(\sqrt{ta^2+(1-t)b^2})> F(\sqrt{t}a)+F(\sqrt{1-t}b),
\ee
then the maximum in
\be
\label{maxprob}
M_1\eqdef\sup_{u\in H_0^1(B), \|\nabla u\|_2=1}\int_B F(u)d\mu
\ee
is attained and for any minimizing sequence $u_k$ for \eqref{maxprob} there exists a sequence $\zeta_k\in B$ such that  $u_k\circ\eta_{\zeta_k}\rightharpoonup u\neq 0$.
converges in $H_0^1(B)$ to the point of maximum.
\end{theorem}
In Section~3 we prove
Theorem~\ref{thm:coco}, Theorem~\ref{minimizers}, and a statement on the general structure of bounded sequences in $H_0^1(B)$, Theorem~\ref{abstractcc}, similar to Struwe's global compactness in \cite{StruweGC}) and to a related statement of P.-L. Lions in \cite{PLL-HF} (note also that M\"obius shifts are also involved in existence proof for the Plateau problem, \cite{StruweP}).  In Section 4, the Appendix, we summarize relevant facts about the Poincar\'e disk.

\section{Proof of the invariant Trudinger-Moser inequality}

$W\Subset B$ We start with the following elementary lemma.
\begin{lemma}
 \label{lem:TMS}
Let  $W\subset\R^2$ be an open disk of radius $\frac12$ and let
 \be
 \label{norm}
 \|u\|_W^2\eqdef \int_{W}\left(|\nabla u|^2+\lambda u^2\right)dx,\,\lambda>0.
 \ee
There exists a number $q>0$ such that
\be
\label{TMS}
\sup_{u\in H^1(W),\|u\|_W\le 1}\int_W e^{q u^2}dx<\infty.
\ee
\end{lemma}
\begin{proof}  Let $T$ be an extension operator from $H^1(W)$ into $H_0^1(B)$.
Then \eqref{TMS} follows from
\be
\label{TMS01}
\sup_{u\in H_0^1(B),\|\nabla u\|_2\le T\|}\int_B e^{q u^2}dx<\infty,
\ee
which follows from the Trudinger-Moser inequality whenever $q\le 4\pi/\|T\|^2$
\end{proof}


\begin{lemma}
 \label{lem:TMS+}
Let  $W\subset\R^2$ be an open disk of radius $\frac12>0$ and let
 \be
 \label{norm+}
 \|u\|_W^2\eqdef \int_{W}\left(|\nabla u|^2+\lambda u^2\right)^2dx,\,\lambda>0.
 \ee
Let $q$ be as in Lemma~\ref{lem:TMS}. Then there is a positive constant $C=C(\lambda)$ such that for all $u\in H^1(W)$ satisfying $\|u\|_W<1$,
\be
\label{TMS+}
\int_W (e^{q u^2}-1)dx\le C\frac{\|u\|_W^2}{1-\|u\|_W^2}.
\ee
\end{lemma}
\begin{proof}
Form \eqref{TMS} we have
\be
\label{pe+}
\frac{(q)^n}{n!}\int_W(u/\|u\|_W)^{2n}dx\le C,\;  u\in H^1(W)\setminus\{0\}\,n\in\N,
\ee
and thus
\be
\label{pe+1}
\frac{(q)^n}{n!}\int_Wu^{2n}dx\le C\|u\|_W^{2n},\;  n\in\N.
\ee
Adding the inequalities \eqref{pe+1} over $n\in\N$ and taking into account the assumption $\|u\|_W<1$, we obtain
\eqref{TMS+}.
\end{proof}

\begin{theorem}
 \label{thm:HTM+} The following relation holds true:
\be
\label{HTM+}
\sup_{u\in\dot H^1(\mathbb H^2),\|u\|\le 1}\int_{\mathbb H^2} (e^{4\pi u^2}-1)d\mu<\infty.
\ee
\end{theorem}
\begin{proof}
Consider $\mathbb H^2$ as the Poincar\'e disk $B$.
Let  $W\subset B$ be an open disk or radius $\frac12$ and set the following equivalent Sobolev norm on $W$
\be
\|u\|_{W}^2\eqdef=\int_{W}\left(|\nabla u|^2+\frac{u^2}{(1-|x|^2)^2}\right)dx.
\ee
By  Corollary~\ref{cor:cover} there is a countable set $Z\subset B$ and the number $M\in\N$ be such that the sets $\eta_\zeta(W)$, $\zeta\in Z$, cover $B$ with multiplicity not exceeding $M$.
Let us fix a function $u\in\dot H^1(\mathbb H^2)$ such that $\|u\|_{\dot H^1}\le 1$ and define
\be
Z_u\eqdef\{\zeta\in Z: \|u\circ\eta_\zeta\|_W^2\ge\frac{q}{8\pi}\},
\ee
where $q$ is as in Lemma~\ref{lem:TMS}.
It is easy to see that $Z_u$ contains at most $40\pi M/q$ elements.
Indeed, since the multiplicity of the covering of $B$ by $\eta_\zeta W$, $\zeta\in Z$, is not greater than $M$,  taking into account that the $\dot H^1(\mathbb H^2)$-norm in the Poincar\'e disk model is realized by $\|\nabla u\|_2$, and applying the Hardy's inequality \eqref{lambda1}, we have
$$
\frac{q}{8\pi}(\#Z_u)\le \sum_{\zeta\in Z_u}\|u\circ\eta_\zeta\|_{W}^2\le M \left(\|u\||_{\dot H^1}^2+\int_Bu^2d\mu \right)\le  5 M\|u\||_{\dot H^1}^2\le 5M.
$$
From Lemma~\ref{lem:TMS+}, applied to $\sqrt{\frac{4\pi}{q}}u$  we have, for every $\zeta\in Z\setminus Z_u$,
\be
\label{TMS+a}
\int_{\eta_\zeta W}(e^{4\pi u^2}-1)d\mu\le C\frac{\|u\circ\eta_\zeta\|_{W}^2}{\frac{q}{4\pi}-\|u\circ\eta_\zeta\|_{W}^2}\le \frac{8\pi}{q}C\|u\circ\eta_\zeta\|_{W}^2.
\ee
Adding the inequalities over $\zeta\in Z\setminus Z_u$ we obtain, using again the Hardy's inequality,
\be
\label{TMS+b}
\int_{\bigcup_{\zeta\in Z_u}\eta_\zeta W}(e^{4\pi u^2}-1)d\mu \le C\int_{B}|\nabla u|^2dx+C\int_{B}u^2d\mu\le C\|u\||_{\dot H^1}^2\le C.
\ee
On the other hand, from the usual Trudinger-Moser inequality for $u\circ\eta_\zeta$, with any $ \zeta \in Z_u$, we have
\be
\label{TMS+c}
\int_{\eta_\zeta W}(e^{4\pi u^2}-1)d\mu=
\int_{W}e^{4\pi (u\circ\eta_\zeta)^2}d\mu
\le C\int_{W}e^{4\pi (u\circ\eta_\zeta)^2}dx
\le \int_{B}e^{4\pi (u\circ\eta_\zeta)^2}dx\le C.
\ee
Adding (at most $40\pi M/q+1$) inequalities  \eqref{TMS+b}  and \eqref{TMS+c}, we obtain \eqref{HTM+}.
\end{proof}

{\em Proof of Theorem~\ref{WTM}.} Rewrite \eqref{HTM+} in coordinate form for the Poincar\'e disk. \hfill  $\Box$

\begin{remark}
\label{best4pi} The constant $4\pi$ in \eqref{HTM+} as well as in \eqref{WTM} cannot be replaced by any number $p>4\pi$. Indeed, the integrals in both relations are bounded from below by $\int_B e^{pu^2}dx$, from the Trudinger-Moser inequality, for which the parameter $4\pi$ is optimal.
\end{remark}

\section{Existence of minimizers}
We begin with the proof of the first statement of Subsection~1.3.

{\em Proof Theorem~\ref{thm:coco}.}
 Let us fix $p<4\pi$. Let  $u_k\in H_0^1(B)$, $\|\nabla u_k\|_2=1$, and assume that
$u_k\circ\eta_{\zeta_k}\rightharpoonup 0$  for every sequence $\zeta_k\in B$.
By \eqref{HTM+}, there is a constant $C>0$ such that for all $n=0,1,\dots$,
\be
\label{pe4aa}
\frac{p^n}{n!}\int_{B}u_k^{2n}d\mu\le C \left(\frac{p}{4\pi}\right)^n.
\ee
Then, for every $m\in\N$ and for all $k\in\N$,
\be
\label{pe4a}
\sum_{n\ge m} \frac{p^n}{n!}\int_{B}u_k^{2n}d\mu\le C \left(\frac{p}{4\pi}\right)^m.
\ee
Furthermore, it is easy to see that there exists $\lambda\in(p/4\pi,1)$ such that
\be
\label{pe4ab}
\sum_{n\ge m} \frac{p^n}{n!}\int_{B}u_k^{2n+r}d\mu\le C \lambda^m.
\ee
By Lemma~9.4 of \cite{ccbook}, for every $n=0,\dots,m-1$,
\be
\label{lowN}
\frac{p^n}{n!}\int_{B}u_k^{2n+r}d\mu\to 0.
\ee
Combining \eqref{lowN} with \eqref{pe4ab}, we obtain

\be
\label{pe4b}
\limsup_{k\to\infty}\int_BF(u_k)d\mu\le C\limsup_{k\to\infty}\int_B|u_k|^re^{pu_k^2}d\mu
\le \epsilon,
\ee
and since $\epsilon$ is an arbitrary positive number, \eqref{cocodmu} follows.
\hfill$\Box$

We will need the following version of Brezis-Lieb lemma in presence of a $H_0^1$-bound on a sequence.
\begin{lemma}
\label{BL}
Let $F\in C^1(\R)$ satisfy $|F(s)|\le Cs^2e^{p s^2}$ with some  $C>0$ and $p<4\pi$, and assume that $u_k\in H_0^1(B)$, $u_k\rightharpoonup u$, $\|\nabla u_k\|_2\le 1$. Then
\be
\label{eq:BL}
\int_B(F(u_k)-F(u_k-u)-F(u))d\mu\to 0.
\ee
\end{lemma}
\begin{proof} The notation of norm in this proof refers to the gradient norm $\|\nabla u\|_2$ on $B$.
Note that $u_k\to u$ almost everywhere in $B$ and that $\limsup\|u_k-u\|^2=\limsup\|u_k\|^2-\|u\|^2\le 1$.
Let $M>0$ and define
$$
F_M(s)\eqdef\left\lbrace\begin{array}{cc}
              F(s) & \text{ for }|s|>M, \\
              0     & \text{ for } |s|\le M.
            \end{array}\right.
$$

Set $G_M=F-F_M$. Then $G_M$ is a bounded function and therefore, by Lebesgue convergence theorem,
\be
\label{GM}
\int_B(G_M(u_k)-G_M(u_k-u)-G_M(u))d\mu\to 0.
\ee
Fix two numbers $q,r$ such that $p<q<r<4\pi$ and note  that $|F_M(s)|\le s^2e^{q s^2}e^{-(q-p)M^2}\le
C (e^{r s^2}-1)e^{-(q-p)M^2}$.
Then
\be
\left|\int F_M(u_k)d\mu\right|\le  C e^{-(q-p)M^2}\int_B(e^{r u_k^2}-1)d\mu\le C e^{-(q-p)M^2}
\ee
by \eqref{WTM}, with analogous estimates when $u_k$ is replaced by $u$, resp. $u-u_k$.
In the latter case \eqref{WTM} is applied to $u-u_k$ if  $u_k\to u$ in norm, and to  $(u-u_k)/\|u-u_k\|$ otherwise.
From here and \eqref{GM} we conclude that
$$
\limsup\left|\int_B(F(u_k)-F(u_k-u)-F(u))d\mu\right|\le  C e^{-(q-p)M^2}.
$$
Since the number $M$ is arbitrary,
\eqref{eq:BL} follows.
\end{proof}

{\em Proof of Theorem~\ref{minimizers}.}
 For the length of this proof the notation of norm, unless otherwise specified, refers to the gradient norm $\|\nabla \cdot\|_2$.
Let $u_k\in H_0^1(B)$ be such that $\|u_k\|\to 1$ and $\int_B F(u)d\mu\to M_1$.
Consider the following family of problems that extends \eqref{maxprob}:
\be
\label{maxprob-t}
M_t\eqdef\sup_{u\in H_0^1(B), \|u\|^2=t}\int_B F(u)d\mu,\;t\in[0,1].
\ee
If  $u_k$ is a maximizing sequence then so is $u_k\circ\eta_{\zeta_k}$ for any  sequence $\zeta_k\in B$. If $u_k\circ\eta_{\zeta_k}\rightharpoonup 0$ for any sequence $\zeta_k$, then by Theorem~\ref{thm:coco} we have $\int_B F(u_k)d\mu\to 0$, a contradiction since $\sup F>0$ implies $M_1>0$. Thus we choose a sequence $\zeta_k\in B$ such that   $u_k\circ\eta_{\zeta_k}\rightharpoonup u\neq 0$.

By the standard scalar product calculations we have
\be
\label{energy}
\|u\|^2+\|u_k-u\|^2=1,
\ee
while by Lemma~\ref{BL},
\be
\label{mass}
\int_B F(u)d\mu+\int_B F(u_k-u) d\mu\to M_1.
\ee
Let $t=\|u\|$. Then from $\eqref{mass}$ follows
\be
\label{mass2}
M_t+M_{1-t}\ge M_1.
\ee
An elementary argument using the well-known property of the gradient norm,
$$\|v_0\|\le 1,\;\|v_1\|\le 1,\; v_t=\sqrt{tv_1^2+(1-t)v_0^2} \Rightarrow \|v_t\|\le 1,\; t\in(0,1),$$
shows that \eqref{mass2} contradicts \eqref{convexity} unless  $t=1$ or $t=0$.
The latter case has been, however, ruled out. Consequently, $\|u\|=1$, $\int_B F(u)d\mu=M$, and $u_k\to u$
in $H_0^1(B)$. The theorem is proved.\hfill$\Box$

The notation $\cw 0$ in the theorem below is a shorthand for the convergence in the sense of \eqref{cw}, which by Theorem~\ref{thm:coco} implies convergence in the sense of
\eqref{cocodmu} and, in particular,  convergence in $L^p(B,d\mu)$ for any $p\in[1,\infty)$.

\begin{theorem}
\label{abstractcc}
Let $u_{k}  \in H_0^1(B)$ be a  bounded sequence. Then there exists
$w^{(n)}  \in H$, $\zeta_{k} ^{(n)}  \in B$, $k,n \in \N$,  such that  for
a renumbered subsequence
\begin{gather}
\label{separates}
\zeta_k^{(1)}=0,\; \eta_{\zeta_{k} ^{(n)}} (\zeta_{k} ^{(m)})\to \partial B
\mbox{ for } n \neq m,
\\
w^{(n)}=\wlim  u_k\circ\eta_{\zeta_{k} ^{(n)}}^{-1}
\\
\label{norms}
\sum_{n \in {\N}} \|\nabla w ^{(n)}\|_2^2 \le \limsup \|\nabla u_k\|_2^2
\\
\label{BBasymptotics}
u_{k}  - \sum_{n\in{\N}}  w^{(n)}\circ\eta_{\zeta_{k} ^{(n)}} \cw 0.
\end{gather}
\end{theorem}
\begin{proof}
This theorem is an application of Theorem~3.1 in \cite{ccbook} to sequences in
$\dot H^1(\mathbb H^2)$ equipped with the M\"obius shifts. Conditions of  that theorem have been verified for the case of actions of isometries on cocompact (or grid-periodic) manifolds, which includes hyperbolic spaces, in Lemma~2.9, \cite{FiesTin}. Relation~\eqref{separates} is based on Remark~9.1 (a) of \cite{ccbook}
\end{proof}

\section{Appendix}
We summarize here some known definitions and facts concerning the Poincar\'e disk model of hyperbolic space.
For reference see \cite{Lang} and  \cite{Ratcliffe}.
Poincar\'e disk is a coordinate representation of the hyperbolic space $\mathbb H^2$, consisting of the unit disk $B\subset\R^2$ equipped with the metric $g_{i,j}=\frac{1}{(1-|x|^2)^2}\delta_{i,j}$, $i,j=1,2$.
The Riemannian measure $\mu$ on $\mathbb H^2$ is given in the Poincar\'e disk model by $d\mu=\frac{dx}{(1-|x|^2)^2}$.
The quadratic form of Laplace-Beltrami operator on $\mathbb H^2$ in the Poincar\'e disk model evaluates as
$\int_B |\nabla u|^2 dx$.
The maps $\eta_\zeta: B\to B$, $\eta_\zeta(z)=\frac{z-\zeta}{1-\bar \zeta z}$, $\zeta\in B$, are conformal isomorphisms of B as well as isometries of the Poincar\'e disk. Consequently, $\int_B |\nabla u|^2 dx$ and $\int F(u)d\mu$ are preserved under transformations $u\mapsto u\circ\eta_\zeta$.
The following version of Hardy's inequality holds true for all $u\in H_0^1(B)$, or, in invariant notations,
$u\in \dot H^1(\mathbb H^2)$ (see \cite{AA} or \cite{ManciniSandeep}):
\be
\label{lambda1} \int_B |\nabla u|^2 dx =\int_{\mathbb H^2} |\nabla_{\mathbb H^2} u|^2 d\mu\ge \frac14\int_{\mathbb H^2}u^2 d\mu=\int_B u^2\frac{dx}{(1-|x|^2)^2}.
\ee

The following lemma is well known (for example, it is a trivial modification of Lemma~2.3 from \cite{FiesTin}).
We give the proof of it for the sake of completeness.

\begin{lemma}
 \label{cover} Let $U\Subset B$ be an open set, let
\be
\label{def:V}
V:=\bigcup_{\zeta\in B: \eta_\zeta U\cap U\neq\emptyset}\eta_\zeta U.
\ee
and let $W\Subset B$ be any open set, relatively compact in $B$, that contains $V$.
There exist a number $M\in\N$, and a countable set $Z\subset B$ such that the sets $\{\eta_\zeta W\}_{\zeta\in Z}$, cover $B$ with multiplicity not exceeding $M$ and the sets $\{\eta_\zeta U\}_{\zeta\in Z}$ are mutually disjoint.
\end{lemma}
\begin{proof} We show first that if $Z\subset B$ is a set such that the sets $\{\eta_\zeta U\}_{\zeta\in Z}$ are mutually disjoint, and the sets $\{\eta_\zeta W\}_{\zeta\in Z}$ cover $B$, then the latter collection has a uniformly finite multiplicity.
Let $V_r(x)$ denote a geodesic ball of radius $r>0$ centered at $x\in B$, and note that $\mu(V_r(x))$ is independent of $x\in B$. Let $R>0$ and $x_0\in B$ be such that  $W\subset V_R(x_0)$.
If $\eta_\zeta W$ intersects $\overline{V_r(x)}$, then
$\eta_\zeta W \subset V_{r+2R} (x)$; but since the sts  $\eta_\zeta U$, $\zeta\in Z$,  are disjoint,
and  $\eta_\zeta W\supset \eta_\zeta V\supset \eta_\zeta U$, $\zeta\in Z$,
that can be true for at most $\frac{\mu_{r+2R}}{\mu (U)}$ values of $\zeta \in Z$.
Therefore the number of $\zeta\in Z$ such that the set $\eta_\zeta W$ contains the point $x$ does not exceed $\frac{\mu_{r+2R}}{\mu (U)}$.

Now let us construct the set $Z$.

Since every Riemannian manifold is paracompact, and once we observe that $\eta_\zeta(\{0\})=-\zeta$, so that $\bigcup_{\zeta\in B}\eta_\zeta(0)=B$, there exists a subset
$Z_0 \subset B$, such that $\{\eta_\zeta V\}_{\zeta\in B}$
is a locally finite cover of $B$.
Indeed, we find first of all a locally finite refinement of the cover $\{\eta_\zeta V\}_{\zeta\in B}$,
which via the refinement map determines a subcover, which also is locally finite due
to the fact that all covering sets $\eta_\zeta V$ have the same finite geodesic diameter.

By induction we define subsets $Z_k= A_k \cup B_k \subset B$
such that the number of elements in $A_k$ equals $k$ and
$$
B=\bigcup_{\zeta \in A_k} \eta_\zeta V \cup \bigcup_{\zeta \in B_k} \eta_\zeta U\ ,
$$
and $\eta_{\zeta_1} U\cap \eta_{\zeta_2}U = \emptyset$ for any $\zeta_1 \in A_k, \zeta_2 \in
Z_k, \zeta_1 \not= \zeta_2$.
Furthermore $A_k \subset A_{k+1}$ for all $k$, while
$B_k \supset B_{k+1}$ with $\cap_{k=0}^\infty B_k= \emptyset$.
Since the cover $\{\eta_\zeta U\}_{\zeta \in Z_0}$ was locally finite, the latter implies
that any compact set
$K \Subset B$ is contained in $\bigcup_{\zeta \in A_k} \eta_\zeta V$ for sufficiently large
$k$. Finally take $Z:= \bigcup_{k=0}^\infty A_k$.
Begin with $A_0:= \emptyset, B_0:=Z_0$.
Let $\{ \zeta_j\}_{j\in\N}$ be an enumeration of $Z_0$.
Assuming that $A_k$, $B_k$ have already been constructed, let us construct $A_{k+1}$, $B_{k+1}$. Let
$j_k=\min\{j:\;\zeta_j\in B_k\}$. Set
$A_{k+1}:=A_k \cup \{ \zeta_{j_k} \}$ and let
$B_{k+1}:= \{ \zeta \in B_k; \eta_\zeta U \cap \eta_{\zeta_{j_k}}U= \emptyset \}$.
\end{proof}
\begin{corollary}
\label{cor:cover} Let $W\Subset B$ be an open set.
There exist a number $M\in\N$, and a countable set $Z\subset B$ such that the sets $\{\eta_\zeta W\}_{\zeta\in Z}$ cover $B$ with multiplicity not exceeding $M$.
\end{corollary}
\begin{proof}
Let $\epsilon>0$ and $x\in W$ be such that the geodesic ball $V_{3\epsilon}(x)\Subset W$. The corollary is immediate from Lemma~\ref{cover}  with $U=V_\epsilon(x)$ once we note  that the set \eqref{def:V} is contained in $V_{3\epsilon}(x)$. Indeed, let $y\in V$. Then there exist $\zeta\in Z$ and $z\in V_\epsilon(x)\cap V_\epsilon(\eta_\zeta x) $ such that $y\in\eta_\zeta V_\epsilon(x)$. Then, by the triangle inequality for the geodesic distance, $d(x,y)\le d(x,z)+d(z,\eta_\zeta x)+d(\eta_\zeta x,y)<3\epsilon $.
\end{proof}

\end{document}